\documentclass{birkmult}
 \newtheorem{theorem}{Theorem}[section]
 \newtheorem{corollary}[theorem]{Corollary}
 \newtheorem{lemma}[theorem]{Lemma}

 \theoremstyle{definition}
 \newtheorem{definition}[theorem]{Definition}
 \theoremstyle{remark}
 \newtheorem{remark}[theorem]{Remark}
 \newtheorem{example}{Example}
 \numberwithin{equation}{section}

\DeclareMathOperator{\re}{Re}

\DeclareMathOperator{\ord}{ord}

\DeclareMathOperator{\Res}{\mathcal{R}}

\def\Com#1{\mathbb{C}^{#1}}

\def\I{\mathrm{i}}

\def\D{{\mathbb D}}

\def\C{{\mathbb C}}

\begin{document}

\title{On the exponential transform of lemniscates}

\author{Bj\"orn Gustafsson}

\address{Mathematical department\\ KTH}

\email{gbjorn@kth.se}

\author{Vladimir Tkachev}

\address{Mathematical Department, KTH}

\email{tkatchev@kth.se}

\date{\today}

\thanks{Paper supported by Swedish Research Council, the Swedish Royal Academy
of Sciences and the G\"oran Gustafsson Foundation.}

\maketitle


{\it In memory of Julius Borcea}


\begin{abstract}
It is known that the exponential transform of a quadrature domain
is a rational function for which the denominator has a certain separable form.
In the present paper we show that the exponential transform of lemniscate
domains in general are not rational functions, of any form. Several examples are given to 
illustrate the general picture. The main tool used is that of 
polynomial and meromorphic resultants.
\end{abstract}

\section{Introduction}

The \textit{exponential transform} \cite{Carey-Pincus74},
\cite{Putinar98}, \cite{Gustafsson-Putinar98} of a domain $\Omega$
in the complex plane is the function of two complex variables
$z,w\in \Com{}\setminus \overline{\Omega}$ defined by
\begin{equation}\label{defexp}
E_\Omega (z,w)=\exp[\frac{1}{2\pi\I}\int_\Omega \frac{d\zeta}{\zeta
-z}\wedge \frac{d\bar{\zeta}}{\bar{\zeta} -\bar {w}}].
\end{equation}
A bounded domain $\Omega\subset\C$ is called a \textit{quadrature
domain} \cite{Aharonov-Shapiro76}, \cite{Sakai82}, \cite{Shapiro92},
\cite{Gustafsson-Shapiro05} if there exist finitely many points
$z_k\in \Omega$ (the \textit{nodes} of $\Omega$) and coefficients
$c_{kj}\in \C$ ($k=1,\dots, N$, say) such that
\begin{equation}\label{QD1}
\int_\Omega h\, dxdy =\sum_{k=1}^N \sum_{j=1}^{s_k} c_{kj} h^{(j-1)}(z_k)
\end{equation}
for every integrable analytic function $h$ in $\Omega$.
The number $d=\sum_{k=1}^Ns_k$ is called the order of $\Omega$. The
simplest example of a quadrature domain is any disk, for which the
center is the only node ($N=1$).

In 1994 M.~Putinar \cite{Putinar94} (see also \cite{Putinar96})
proved that a bounded domain $\Omega$ is a quadrature domain if and
only if  $E_{\Omega}(z,w)$ for large values of $z$ and $w$ is a
rational function of the form
\begin{equation}\label{e1}
E_\Omega(z,w)=\frac{Q(z, w)}{P(z)\overline{P(w)}},
\end{equation}
where $P(z)$ is an ordinary polynomial and $Q(z,w)$ is a Hermitean
polynomial, i.e., a polynomial in $z$ and $\bar w$ satisfying
$Q(w,z)=\overline{Q(z,w)}$. Moreover, when (\ref{e1}) holds near infinity it
remains valid in all of $(\C\setminus\overline{\Omega})^2$. In addition,
$Q(z,z)=0$ is the defining equation of the boundary
$\partial\Omega$, except for a finite number points, and the zeros
of $P$ are exactly the nodes $z_k$ in (\ref{QD1}). Thus the shape of
a quadrature domain is completely determined by $Q$, or by $E$.

Putinar's result does not exclude that there exist other domains
than quadrature domains for which the exponential transform is a
rational function, then of a more general form than (\ref{e1}) or
only in certain components $(\C\setminus\overline{\Omega})^2$. There indeed do
exist such domains, for example circular domains and domains between
two ellipses. However, all known examples are multiply connected domains which 
are obtained by relatively trivial
modifications of quadrature domains. Therefore the question arises
whether there exist domains definitely beyond the category of
quadrature domains for which the exponential transform is rational in part or 
all of the complement,

Looking from the other side, any domain having rational exponential
transform (in all parts of the complement) necessarily has an
algebraic boundary, because of the boundary behavior of the
exponential transform. The simplest type of domains having an
algebraic boundary, but being definitely outside the scope of
quadrature domains, are lemniscate domains. The relatively modest
main result of the present paper says that for certain types of
lemniscate domains the exponential transform is not a rational
function.

\begin{theorem}\label{th:mm}
Let $\Omega$ be a bounded domain such that there is a $p$-valent
proper rational map $f:\Omega\to \mathbb{D}$ with
$f(\infty)=\infty$. Let $n=\deg f$ be the degree of $f$ as a
rational function. Then, if $n>p$ the exponential transform
$E_\Omega(z,w)$ is not a rational function for $z$ and $w$ in the
unbounded component of $\C\setminus\overline{\Omega}$.
\end{theorem}

Since almost every point of $\mathbb{D}$ has $n=\deg f$ preimages in
total under $f$ and $\Omega$ is assumed to contain only $p<n$ of
these, the assumptions imply that $f^{-1}(\mathbb{D})$ has several
components and that $\Omega$ is only one of them.

A typical situation when Theorem~\ref{th:mm} is applicable is when
$f$ is a rational function of degree $n\geq 2$, which sends infinity
to itself and has only simple zeros. Then for $\epsilon$ small
enough, the set $\{z:|f(z)|<\epsilon\}$ consists exactly of $n$ open
components $\Omega_k$, each containing inside a single zero of $f$.
It follows that $\frac{1}{\epsilon}f|_{\Omega_k}$ is a univalent map
of $\Omega_k$ onto $\mathbb{D}$ and by Theorem~\ref{th:mm} the
exponential transform of $\Omega_k$ is non-rational. This example
can be easily generalized to a wider class of rational functions and
multiplicities.

Besides the above result (Theorem~\ref{th:mm}), the paper contains
methods which may give further insights into the nature of the
exponential transform  and its connections to resultants. We also
give some examples and, in particular, a detailed analysis of the exponential
transform and complex moments for the Bernoulli lemniscate.

As for the organization of the paper, in the first sections we
review some facts about exponential transforms, quadrature domains
and meromorphic resultants which will be needed in the proof of the
main result. The proof of Theorem~\ref{th:mm} is given in
Section~\ref{sec:main}. A few simple examples are given in
Section~\ref{sec:ex} and a more elaborate example, on the Bernoulli
lemniscate, in Section~\ref{sec:app}.

Some related recent results on lemniscates are contained in
\cite{ebenfelt-khavinson-shapiro10} and
\cite{khavinson-mineev-putinar-teodorescu09}.


\section{The exponential transform}

Here we list some basic properties of the exponential  transform. A
full account with detailed proofs may be found in
\cite{Gustafsson-Putinar98}. Even though the definition
(\ref{defexp}) of the exponential transform makes sense for all
$z,w\in\C$ we shall in this paper only study it for
$z,w\in\C\setminus\overline{\Omega}$.
On the diagonal $w=z$ we have $E_\Omega(z,z)>0$ for
$z\in \Com{}\setminus \overline{\Omega}$ and
\begin{equation}\label{exex}
\lim_{z\to z_0} E_\Omega(z,z)=0\;
\end{equation}
for almost all $z_0\in \partial \Omega$. Notice that this property
allows to recover the boundary $\partial\Omega$ from
$E_\Omega(z,w)$.

The exponential transform is Hermitian symmetric:
\begin{equation}\label{et1}
E_\Omega(w,z)=\overline{E_\Omega(z,w)}.
\end{equation}
Expanding the integral in the definition of $E_\Omega(z,w)$ in power
series in $1/\bar w$ gives
\begin{equation}\label{Cauchy0}
E_\Omega(z,w)=1-\frac{1}{\bar w}C_\Omega(z)+ \mathcal{O}(\frac{1}{|w|^2})
\end{equation}
as $|w|\to \infty$, with $z\in\mathbb{C}\setminus \overline{\Omega}$ fixed. Here
$$
C_\Omega(z)= \frac{1}{2\pi \I}\int_{\Omega} \frac{d\zeta\wedge
d\bar\zeta }{\zeta-z}
$$
is the Cauchy transform of $\Omega$.

For explicit evaluations of the exponential transform one can use
its representation in terms of the complex moments of $\Omega$:
$$
M_{pq}(\Omega)=-\frac{1}{2\pi \I}\int\limits_{\Omega} z^p\bar z^q
dz\wedge d\bar z, \quad p,q\geq 0.
$$
Namely, for $z,w$ large enough,
\begin{equation}\label{infact}
E_{\Omega}(z,w)=\exp(-\sum_{p,q=0}^\infty
\frac{M_{pq}(\Omega)}{z^{p+1}\bar w^{q+1}}).
\end{equation}
We shall demonstrate how this can be used in Section~\ref{sec:app}.
For the round disk $\mathbb{D}(a,r)=\{\zeta\in \mathbb{C}:
|\zeta-a|<R\}$ the exponential transform is (see
\cite{Gustafsson-Putinar98})
\begin{equation}\label{et3}
E_{\mathbb{D}(a,r)}(z,w)=
\begin{cases}
1-\frac{R^2}{(z-a)(\bar w-\bar a)}
&\text{for $z,w\in \mathbb{C}\setminus \overline{\mathbb{D}(a,r)}$},\\

-\frac{\bar z-\bar w}{\bar w-\bar a}
&\text{for $z\in \mathbb{D}(a,r)$, $w\in \mathbb{C}\setminus \overline{\mathbb{D}(a,r)}$},\\
\frac{z-w}{z-a}
&\text{for $z\in \mathbb{C}\setminus \overline{\mathbb{D}(a,r)}$, $w\in \mathbb{D}(a,r)$},\\
\frac{|z-w|^2}{R^2-(z-a)(\bar w-\bar a)}&\text{for $z,w\in \mathbb{D}(a,r)$}.\\
\end{cases}
\end{equation}
Here we have listed the values in all $\C^2$ because we shall need
them later to compute the exponential transform for circular
domains.


\section{Quadrature domains and lemniscates}

In this paper we shall mean by a \textit{lemniscate} $\Gamma$ a
plane algebraic curve given by an equation $|f(z)|=1$, where $f(z)$
is a rational function which preserves the point of infinity:
$f(\infty)=\infty$. Hence any lemniscate is given by an equation
\begin{equation}\label{lem1}
|A(\zeta)|=|B(\zeta)|,
\end{equation}
where $A$ and $B$ are relatively prime polynomials, with $B$ assumed
to be monic (that is, with leading coefficient equal to one) and
$n=\deg A>m=\deg B$. The rational function $f$ then is
$f(\zeta)=A(\zeta)/B(\zeta)$ and, as usual, the degree of $f$ is
defined by
$$
\deg f=\max\{\deg A,\deg B\}=n.
$$

Under these  conditions, the algebraic curve $\Gamma$ is  the
boundary of the (bounded) sublevel set $\Omega=\{\zeta:
|f(\zeta)|<1\}$. The latter open set may have several components,
and any such component will be called a \textit{lemniscate domain}.
Notice that $f$ is a proper $n$-to-$1$ holomorphic map of $\Omega $
onto the unit disk:
$$
f: \Omega\to \mathbb{D}=\{z:\; |z|<1\}.
$$
The unit disk itself is the simplest lemniscate domain,  with
$f(\zeta)=\zeta$. When $\deg B=0$ (that is, $B\equiv 1$) we arrive
at the standard definition of a polynomial lemniscate (cf.
\cite[p.~264]{Hille}).

Lemniscates and quadrature domains in the complex plane can be
thought of as dual classes of objects. Indeed,  it is well-known
that any quadrature domain has an algebraic boundary (see
\cite{Aharonov-Shapiro76}, \cite{Gustafsson83}, \cite{Shapiro92},
\cite{Gustafsson-Shapiro05},  \cite{Varchenko-Etingof94}), the
boundary being (modulo finitely many points) the full real section
of an algebraic curve:
\begin{equation}\label{lamma}
\partial \Omega = \{z\in \mathbb{C}: Q(z,z)=0\},
\end{equation}
where $Q(z,w)$ is an irreducible Hermitian polynomial, the same as
in (\ref{e1}). Moreover, the corresponding complex algebraic curve
(essentially $\{(z,w)\in \mathbb{C}^2: Q(z,w)=0\}$) can be naturally
identified with the Schottky double $\widehat{\Omega}$ of $\Omega$
by means of the Schwarz function $S(z)$ of $\partial\Omega$. The
latter satisfies $S(z)=\bar{z}$ on $\partial\Omega$ and is, in the
case of a quadrature domain, meromorphic in all $\Omega$.

It is shown in \cite{Gustafsson-Putinar00}
that a quadrature domain of order $d$ is rationally
isomorphic to the intersection of a smooth rational curve of degree
$d$ in the projective space $\mathbb{P}_d (\C)$ and the complement
of a real affine ball. More precisely, for any quadrature domain
$\Omega$ its defining polynomial $Q(z,w)$ in (\ref{lamma})  admits a
unique representation of the kind:
\begin{equation}\label{perp}
Q(z,z)=|P(z)|^2-\sum_{i=0}^{d-1}|Q_i(z)|^2,
\end{equation}
where $P(z)=\prod_{k=1}^N (z-z_k)^{s_k}$ is a monic polynomial of
degree $d$, the leading coefficients of polynomials $Q_i$ are
positive and $\deg Q_i=i$.

Notice that (\ref{perp}) means that the equation for the boundary of
a quadrature domain is
\begin{equation}\label{perp1}
|P(z)|^2=\sum_{i=0}^{d-1}|Q_i(z)|^2,
\end{equation}
which reminds of the defining equation for a lemniscate
(\ref{lem1}). However, the difference in the number of terms in
(\ref{perp1}) and (\ref{lem1}) makes the generalized lemniscates
(\ref{perp1}) (in terminology of M.~Putinar \cite{PutLem}) much
different from the standard lemniscates defined by (\ref{lem1}). For
instance, the exponential transform of a lemniscate domain is no
more a rational function as we shall see later.

Another point which relates lemniscates and quadrature domains to
each other is the following. Recall that for a \textit{simply
connected} bounded domain, P.~Davis \cite{Davis74} and D.~Aharonov
and H.~S.~Shapiro \cite{Aharonov-Shapiro76} proved that $\Omega$ is
a quadrature domain if and only if $\Omega=f(\mathbb{D})$, where $f$
is a rational uniformizing map from the unit disk $\mathbb{D}$ onto
$\Omega$. This property can be thought as dual to the definition of
a lemniscate given above. Indeed, a simply connected quadrature
domain is an image of the unit disk $\mathbb{D}$ under a (univalent
in $\mathbb{D}$) rational map $f$, while a lemniscate is a preimage
of the unit disk under a (not necessarily univalent) rational map
$g$:
\begin{equation*}
\begin{split}
\mathbb{D}\quad & \stackrel{f}{\longrightarrow} \quad \text{a quadrature domain} \\
\text{a lemniscate domain} \quad &\stackrel{g}{\longrightarrow} \quad  \mathbb{D}.\\
\end{split}
\end{equation*}


\section{Resultants}

The main tool in our proof of the main theorem is the meromorphic and polynomial
resultants. Recall that the (polynomial) resultant  of two  polynomials, $A$ and $B$,
in one complex variable is a polynomial function in the coefficients of $A$, $B$
having the elimination property that it vanishes if and only if $A$ and $B$ have
a common zero \cite{Waerden}, \cite{Gelfand-Kapranov-Zelevinskij}. In terms of the
zeros of  polynomials,
\begin{equation}\label{PQ}
A(z)=A_n\prod_{i=1}^n(z-a_i)=\sum_{i=0}^n A_iz^i ,
\quad B(z)=B_m\prod_{j=1}^m(z-b_j) =\sum_{j=0}^mB_jz^j,
\end{equation}
the resultant (with respect to the variable $z$) is given by the Poisson product formula
\cite{Gelfand-Kapranov-Zelevinskij}
\begin{equation}\label{res1}
\begin{split}
\Res_{z}(A,B)&=A_n^mB_m^n\prod_{i,j} (a_i-b_j) =A_n^m\prod_{i=1}^n B(a_i).
\end{split}
\end{equation}
Alternatively, the resultant is the determinant of the Sylvester matrix:
\begin{equation}\label{sfor}
\Res_{z}(A,B)=\det
\begin{pmatrix}
A_0   & A_1 &A_2 &  \ldots & A_n \\
 & A_0  &A_1&A_2& \ldots & A_n\\
    &  & \ldots &\ldots & \ldots  &  \ldots & \ldots \\
  &  &  &A_0  & A_1 &A_2& \ldots & A_n   \\
   B_0  & B_{1} & \ldots &\ldots & B_m \\
 & B_0 & B_{1} &\ldots & \ldots& B_m\\
    &  & \ldots &\ldots &  \ldots&\ldots & \ldots \\
   &  &  &  B_0 & B_{1} &\ldots &\ldots & B_m  \\
\end{pmatrix}.
\end{equation}
It follows from the above definitions that $\Res_{z}(A,B)$  is
skew-sym\-met\-ric and multiplicative:
\begin{equation}\label{mulmul}
\begin{split}
\Res_{z}(A,B)&=(-1)^{mn}\Res_{z}(B,A), \\
\Res_{z}(A_1A_2,B)&=\Res_{z}(A_1,B)\Res_{z}(A_2,B).
\end{split}
\end{equation}
Conjugating the identity in (\ref{res1}) we get
\begin{equation}\label{conjj}
\overline{\Res_z(A(z),B(z))}=\Res_{\bar z}(\overline{A(z}),\overline{B(z})).
\end{equation}

The authors introduced in \cite{GT071} a notion of the meromorphic
resultant of two meromorphic functions on an arbitrary compact
Riemann surface. Here we shall not need this concept in its full
generality, but for our further goals it will be useful to recall some
facts in the case of the Riemann sphere $\mathbb{P}_1(\C)$.

For two  rational functions $f(z)$ and $g(z)$ the  number
\begin{equation}
\label{gagb1new}
\Res^*(f,g)=\prod_{i} g(a_i)^{n_i},
\end{equation}
when defined, is called the \textit{meromorphic resultant} of $f$
and $g$. Here $\sum_{i}n_ia_i$ is the divisor of $f$. This resultant
is symmetric and multiplicative. An essential difference between the
meromorphic resultant and the polynomial one is that the latter
depends merely on the divisors of $f$ and $g$. If
$f(z)=\frac{A_1(z)}{A_2(z)}$ and $g(z)=\frac{B_1(z)}{B_2(z)}$ are
the polynomial representations then we have the following explicit
formula:
\begin{equation}\label{def1}
\Res^*(f,g)=f(\infty)^{\ord_\infty(g)}g(\infty)^{\ord_\infty(f)}\cdot
\frac{\Res(A_1,B_1)\Res(A_2,B_2)}
{\Res(A_1,B_2)\Res(A_2,B_1)},
\end{equation}
where, generally speaking, $\ord_a(f)$ is the order of $f$ at the
point $a$, that is the integer $m$ such that, in terms of a local
variable $z$ at $a$,
\begin{equation*}
f(z)=c_m(z-a)^m + c_{m+1}(z-a)^{m+1}+\dots \quad {\rm with} \quad c_m\ne 0.
\end{equation*}

M.~Putinar has shown, \cite[Theorem~4.1]{Putinar96}, that if
$f:\Omega_1\to \Omega_2$ is rational and univalent then
$E_{\Omega_2}$ is of separable form (\ref{e1}) provided
$E_{\Omega_1}$ is on such a form. We shall need this fact in the
following more general form.

\begin{theorem}[\cite{GT071}, Theorem~8]\label{TH1}
Let $\Omega_i$, $i=1,2$, be two bounded open sets in the complex
plane and let $f$ be a proper $n$-valent rational function which
maps $\Omega_1$ onto $\Omega_2$. Assume that $E_{\Omega_1}(u,v)$ is
a rational function (more precisely, is the restriction to
$(\C\setminus \overline{\Omega_1})^2$ of a rational function). Then,
for all $z,w\in\Com{}\setminus\overline{\Omega}_2$,
\begin{equation}\label{www}
\begin{split}
E_{\Omega_2}(z,w)^n=\Res^*_\xi (f(\xi)-z,\Res^*_{\bar\eta}
(\overline{f(\eta)}-\bar w,E_{\Omega_1}(\xi,\eta))),
\end{split}
\end{equation}
and this is also (the restriction of) a rational function.
\end{theorem}
Another, and perhaps more striking, way to write (\ref{www}) is
$$
E_{\Omega_2}(z,w)^n=E_{\Omega_1}((f-z),(f-w)),
$$
where $(f-z)$, $(f-w)$ denote the divisors of $f(\zeta)-z$,
$f(\zeta)-w$ (as functions of $\zeta$) and the right member refers
to the multiplicative action of $E_{\Omega_1}$ on these divisors.
See \cite{GT071}, in particular Theorem~8, for further details.


\section{Examples and remarks}\label{sec:ex}

Here we give some examples showing that the exponential transform of a multiply connected domain 
may be rational only in some components of the complement, and also 
that it can be rational in all components of the complement but be represented by different 
rational functions in different components. However, we do not know of any domain,
outside the class of quadrature domains, for which the exponential transform is given by 
one and the same rational function everywhere in the complement. 

The final example is supposed to explain, from one point of view,
why lemniscates are fundamentally different from quadrature domains. 

\begin{example}\label{ex1}
For the annulus ${A}(r,R)=\{z\in\C:r<|z|<R\}$ we get, by
using (\ref{et3}),
$$
E_{{A}(r,R)}(z,w)=\frac{E_{\mathbb{D}(0,R)}(z,w)}{E_{\mathbb{D}(0,r)}(z,w)}
=\left(\frac{z\bar w-R^2}{z\bar w-r^2}\right)^\epsilon,
$$
where
$$
\epsilon=
\begin{cases}
1 \quad &{\rm if\,\,}z,w\in \Com{}\setminus
\overline{\mathbb{D}(0,R})\\
-1 \quad &{\rm if\,\,}z,w\in \mathbb{D}(0,r)\\
0 \quad &{\rm if\,\,}z\in \Com{}\setminus
\overline{\mathbb{D}(0,R}),\,w\in \mathbb{D}(0,r)\,\,{\rm
or\,\,vice\,\,versa}.
\end{cases}
$$
Notice that both numerator and denominator are irreducible. In
particular, the annulus is no longer a quadrature domain.

More generally, any domain $\Omega$ bounded by circles has an
exponential transform which is rational in each component of
$(\C\setminus\overline{\Omega})^2$. Indeed, such a domain can be
written
$$
\Omega=\D(a_0,r_0)\setminus\cup_{i=1}^n \overline{\D(a_i, r_i)},
$$
where the $\overline{\D(a_i, r_i)}$ are disjoint subdisks of
$\D(a_0,r_0)$, and since
$$
E_\Omega(z,w)=\frac{E_{\D(a_0,r_0)}(z,w)}{E_{\D(a_1,r_1)}(z,w)\cdots
E_{\D(a_n,r_n)}(z,w)}
$$
the assertion follows immediately from (\ref{et3}).

It should be noted in the present example that $E_\Omega(z,w)$ is
represented by different rational functions in different components
of $(\C\setminus\overline{\Omega})^2$.
\end{example}

\begin{example}
If $D_1$, $D_2$ are quadrature domains with $\overline{D_1}\subset
D_2$, then the exponential transform of
$\Omega=D_2\setminus\overline{D_1}$ is rational in the exterior
component of $(\C\setminus\overline{\Omega})^2$, but generally not
in the other components. The first statement
follows immediately from (\ref{e1}) and the second statement can be
seen from expressions for $E_D (z,w)$ given in
\cite{Gustafsson-Putinar98}. For example, inside a quadrature domain
$D$ the exponential transform is of the form
$$
E_D(z,w)=\frac{|z-w|^2 Q(z,w)}{(z-\overline{S(w)})(S(z)-\bar
w)P(z)\overline{P(w)}} \quad (z,w\in D),
$$
where $S(z)$ is the Schwarz function of $\partial D$. 
When forming
$$
E_\Omega (z,w)=\frac{E_{D_2}(z,w)}{E_{D_1}(z,w)}
$$
one sees that in the right member there appears, for $z,w\in D_1$, besides rational functions
also the factor
$$
\frac{(z-\overline{S_1(w)})(S_1(z)-\bar
w)}{(z-\overline{S_2(w)})(S_2(z)-\bar w)},
$$
which is meromorphic in $D_1\times D_1$ but in general not rational
($S_i$ denotes the Schwarz function of $\partial D_i$).

More explicit evidence will be given in the next example, which
discusses the inversion of a two-point quadrature domain, namely the
ellipse.
\end{example}

\begin{example}\label{ex2}

Consider the ellipse
$$
D=\{z=x+iy\in\C: \frac{x^2}{a^2}+\frac{y^2}{b^2}<1\},
$$
where $0<b<a$. Set $c^2=a^2-b^2$, $c>0$. Writing the equation for
the ellipse in terms of $z$ and $\bar z$ and solving for $\bar z$
gives $\bar z=S_\pm (z)$, where
$$
S_{\pm}(z)= \frac{a^2+b^2}{c^2}z\pm\frac{2ab}{c^2}\sqrt{z^2-c^2}.
$$
Here we make the square root single-valued in $\C\setminus [-c,c]$
by taking it to be positive for large positive values of $z$. Then
$S(z)=S_- (z)$ equals $\bar z$ on $\partial D$, hence this branch is
the Schwarz function for $\partial D$.
According to \cite{Gustafsson-Putinar98},
\cite{Gustafsson-Putinar03} we have
$$
E_D(z,w)=
\begin{cases}-\frac{a+b}{a-b} \cdot\frac{z-\overline{S_-(w)}}{\bar w-S_+(z)}
\quad &{\rm for}\,\, z,w\in
\C\setminus D,\\
-\frac{a+b}{a-b}\cdot \frac{z-w}{\bar w-S_+(z)} \quad &{\rm
for}\,\,z\in \C\setminus D,
\,\,w\in D\setminus [-c,c],\\
\frac{a+b}{a-b}\cdot \frac{\bar z-\bar w}{z-\overline{S_+(w)}} \quad
&{\rm for}\,\,
z\in D\setminus [-c,c],\,\,w\in \C\setminus D,\\
\frac{a+b}{a-b}\cdot \frac{(z-w)(\bar z-\bar w)}{(\bar
w-S_+(z))(\bar w-S_-(z))} \quad &{\rm for} \,\,z,w\in D\setminus
[-c,c].
\end{cases}
$$
Explicitly this becomes
$$
E_D(z,w)=
\begin{cases} -\frac{a+b}{a-b}\cdot\frac{c^2
z-(a^2+b^2)\bar w+2ab\sqrt{{\bar w}^2-c^2}}{c^2\bar
w-(a^2+b^2)z-2ab\sqrt{{z}^2-c^2}} \quad &{\rm for}\,\,z,w\in
\C\setminus D,\\
-\frac{(a+b)^2}{c^2\bar w-(a^2+b^2)z+2ab\sqrt{{z}^2-c^2}}\cdot
(z-w)\quad &{\rm for}\,\,z\in \C\setminus D, \,\,w\in
D,\\
\frac{(a+b)^2}{c^2 z-(a^2+b^2)\bar w+2ab\sqrt{{\bar w}^2-c^2}}\cdot
(\bar z-\bar w) \quad
&{\rm for}\,\,z\in D,\,\,w\in \C\setminus D,\\
\frac{(a+b)^2}{c^2z^2+c^2{\bar w}^2-2(a^2+b^2)z\bar w+4a^2b^2}\cdot
|z-w|^2\quad &{\rm for}\,\,z,w\in D,
\end{cases}
$$
where we have replaced $D\setminus [-c,c]$ by $D$, since
the singularities on the focal segment, which are present in $S(z)$,
do not appear in $E_D(z,w)$. (This is a general fact.)

From the above we see that if we have two ellipses, $D_1$ and $D_2$
with $\overline{D_1}\subset D_2$, then the exponential transform
$E_\Omega=E_{D_2}/E_{D_1}$ of $\Omega=D_2\setminus \overline{D_1}$
is rational in $D_1\times  D_1$ but not in the remaining components of
$(\C\setminus\overline{\Omega})^2$. The square roots 
in the above expression for $E_D(z,w)$ will not disappear.

\end{example}

\begin{example}\label{ex:lemniscate}
The following example is supposed to give a partial explanation of
why lemniscate domains do not have rational exponential transforms,
or at least why they are fundamentally different from quadrature
domains.

Consider the lemniscate domain
$$
\Omega=\{z:\C: |z^n-1|<r^n\},
$$
where $r>1$. This is a simply connected domain bounded by the
lemniscate curve
$$
|z^n-1|=r^n.
$$
The domain $\Omega$ is inside this curve, with the usual
interpretation of the word ``inside''. However, from an algebraic
geometric point of view the lemniscate curve has no inside (or
rather, the inside and the outside are the same).

To explain this, consider the corresponding algebraic curve in
$\C^2$ (or, better, in $\mathbb{P}_2({\mathbb{C}})$) obtained by
setting $w=\bar{z}$ in the above equation:
\begin{equation}\label{algebraiccurve}
z^n w^n -z^n - w^n =r^{2n}-1.
\end{equation}
Solving for $w$ gives the Schwarz function for the lemniscate:
$$
S(z)= \sqrt[n]{\frac{z^n-1+r^{2n}}{z^n-1}}.
$$
This is an algebraic function with $n$ branches, which has branch
points at the solutions of $z^n=1$ and $z^n =1-r^{2n}$. The
branching orders at these points are $n-1$, hence the total
branching order is $2n(n-1)$. The Riemann-Hurwitz formula therefore
gives that the genus of the algebraic curve (\ref{algebraiccurve}) is
$$
{\texttt{g}}=1-n +\frac{1}{2}\cdot 2n(n-1)=(n-1)^2.
$$

Now, what makes quadrature domains special among all domains having
an algebraic boundary is that
the Riemann surface $M$ associated to the algebraic curve defining the boundary 
in a canonical way can be
identified with the Schottky double $\hat{\Omega}$ of the domain,
which generally speaking is a completely different Riemann surface.
In particular this requires that the genus of $M$ agrees with the
genus of $\hat{\Omega}$, which is the number of components of
$\partial\Omega$ minus one.

For the above lemniscate curve the genus of the Schottky double is
zero, while the genus of $M$ is ${\texttt{g}}=(n-1)^2>0$. 
One step further, the algebraic curve defines a symmetric Riemann
surface, the involution being $J:(z,w)\mapsto(\bar{w},\bar{z})$, and
the lemniscate curve is the projection under $(z,w)\mapsto z$ of the
symmetry line $L$ (the set of fixed points of $J$) of this symmetric
Riemann surface. As $L$ has only one component and ${\texttt{g}}>0$,
$L$ cannot separate $M$ into two components: $M\setminus L$ is
connected. This is what we mean by saying that the lemniscate curve
has no inside from an algebraic geometric point of view. Both sides
of the lemniscate are the same, when viewed on $M$. For the above
reason we consider the lemniscate to be seriously beyond the
category of quadrature domains.
\end{example}

\begin{remark}
Unfortunately, Theorem~\ref{th:mm} does not apply to the lemniscate
discussed in above because of the assumption $n>p$ in the theorem.

\end{remark}


\section{Auxiliary results}
We begin with a series of auxiliary facts about general rational exponential transforms.
A polynomial of the kind
\begin{equation*}\label{hermi}
\phi(z,w)=\sum_{i,j=0}^d \phi_{ij}z^i\bar w ^j, \qquad
\phi_{ji}= \overline{\phi_{ij}},
\end{equation*}
is called Hermitian. By (\ref{et1}) any rational exponential transform can be
brought to the following form:
\begin{equation}\label{kind}
E_\Omega(z,w)=\frac{\phi(z,w)}{\psi(z,w)},
\end{equation}
where $\phi$ and $\psi$ are relatively prime Hermitian polynomials.
If the variables in the denominator in (\ref{kind}) separate,
\begin{equation}\label{mmmm}
E_\Omega(z,w)=\frac{\phi(z,w)}{\chi(z)\overline{\chi( w)}},
\end{equation}
we call the exponential transform \textit{separable}.

\begin{definition}\label{defreg}
A Hermitian rational function $E(z,w)=\frac{\phi(z,w)}{\psi(z,w)}$
will be called \textit{regular rational in $U\subset \mathbb{C}$} if
$\phi$ and $\psi$ are relatively prime and
\begin{enumerate}
\item[$(i)$]
$\deg_z \phi=\deg_{\bar w} \phi=\deg_z \psi=\deg_{\bar w} \phi$;

\item[$(ii)$] if $d$ is the common value in $(i)$ and
\begin{equation}\label{herma}
\begin{split}
\phi(z, w)&=\phi_d(z)\bar w^d+\ldots +\phi_1(z)\bar w+\phi_0(z),\\
\psi(z, w)&=\psi_d(z)\bar w^d+\ldots +\psi_1(z)\bar w+\psi_0(z),
\end{split}
\end{equation}
then $\phi_d(z)\equiv\psi_d(z)$;
\item[$(iii)$]
if there exists $z_0\in U$ and two indices $j$ and $k$ such that
$$
\phi_d(z_0)=\ldots=\phi_{k+1}(z_0)=0,\quad  \psi_d(z_0)=\ldots=\psi_{j+1}(z_0)=0,
$$
and $\phi_{k}(z_0)\ne0,$ $\psi_{j}(z_0)\ne0$, then $j=k$ and
\begin{equation}\label{appp}
\phi_{k}(z_0)=\psi_{k}(z_0).
\end{equation}
\end{enumerate}
The common value in $(i)$ is denoted  $\deg E(z,w)$ and called the
degree of $E(z,w)$.
\end{definition}

\begin{remark}
Note that the requirements $(i)-(ii)$ in Definition~\ref{defreg}
identifies a unique monic polynomial $\chi (z)=\phi_d(z)=\psi_d(z)$
and that they mean that $E(z,w)$ is of the form
\begin{equation}\label{remarkseparable}
E_\Omega(z,w)=\frac{\chi(z)\overline{\chi(
w)}+\sum\alpha_{ij}z^i\bar{w}^j}{\chi(z)\overline{\chi( w)}
+\sum\beta_{ij}z^i\bar{w}^j}
\end{equation}
for some Hermitean matrices $(\alpha_{ij})$, $(\beta_{ij})$, $0\leq
i,j\leq d-1$.

\end{remark}

\begin{lemma}
\label{lem:main part} If the exponential transform $E_\Omega(z,w)$
is rational for $z$, $w$ in the unbounded component of $\C\setminus
\overline{\Omega}$ then it is regular rational there.
\end{lemma}

\begin{proof}
The first two properties are straightforward corollaries of the
Hermitian property of $E_\Omega(z,w)$ and the limit relation
(\ref{Cauchy0}).

In order to check $(iii)$ we notice that
$$
E_\Omega(z_0, w)=\frac{\phi_{k}(z_0)\bar w^k+\ldots +\phi_1(z_0)\bar w+\phi_0(z_0)}{
\psi_j(z_0)\bar w^j+\ldots +\psi_1(z_0)\bar w+\psi_0(z_0)}\sim \frac{\phi_{k}(z_0)}{\psi_j(z_0)}\bar w^{k-j}, \quad \text{as $w\to \infty$}.
$$
By virtue of  (\ref{Cauchy0}) we have $j=k$ and $\phi_{k}(z_0)=\psi_{k}(z_0)$.
\end{proof}

Given an arbitrary Hermitian polynomial
$$
\phi(z,w)=\phi_d(z)\bar w^d+\phi_{n-1}(z)\bar
w^{d-1}+\ldots +\phi_1(z)\bar w+\phi_0(z), \qquad
\phi_{d}\not\equiv 0,
$$
we denote by
$$
a(z)=\mathrm{gcd}(\phi_d(z),\phi_{d-1}(z),\ldots,\phi_0(z))
$$
the \textit{monic} (in $z$) polynomial which is the greatest common
divisor of the coefficients of $\phi(z,w)$. We  call $a(z)$ the
\textit{principal divisor} of $\phi(z,w)$. A polynomial $\phi(z,w)$
will be called  \textit{primitive} if $a\equiv 1$. The following
properties are immediate corollaries of the definition.

\begin{lemma}\label{lem:primitive}
(i) A Hermitian polynomial $\phi(z,w)$ is primitive if and only if there is no
$z_0\in\Com{}$ such that $\phi(z_0,w)\equiv 0$ identically in $w$.

(ii) If $a$ is the principal divisor of $\phi(z,w)$ then
\begin{equation}\label{phiphi10}
\phi(z,w)=a(z)\overline{a(w}){\phi}_0(z,w)
\end{equation}
where ${\phi}_0(z,w)$ is a primitive Hermitian polynomial.
Conversely, if $\phi(z,w)$ admits a factorization (\ref{phiphi10})
with ${\phi}_0(z,w)$ primitive then $a(z)$ is (up to normalization)
the principal divisor of $\phi(z,w)$.

\end{lemma}

We shall refer to (\ref{phiphi10}) as to the \textit{principal factorization} of
$\phi(z,w)$.

Let $f(\zeta)=A(\zeta)/B(\zeta)$ be a rational function with $A$ and
$B$ relatively prime polynomials such that $\deg A=n>m=\deg B$ and
define a new polynomial by
\begin{equation}\label{letus}
f_z(\zeta)=A(\zeta)-zB(\zeta), \quad \deg f_z=n.
\end{equation}
It is not hard to check that for any Hermitian polynomial $\phi(\xi,\eta)$, the expression
$$
\Res_\xi(f_z(\xi),\Res_{\bar\eta}(\overline{f_w(\eta)},\phi(\xi,\eta)))
$$
is also a Hermitian polynomial in $z,w$, hence it allows a principal
factorization, which we write as
\begin{equation}\label{thth1}
\Res_\xi(f_z(\xi),\Res_{\bar\eta}(\overline{f_w(\eta)},\phi(\xi,\eta)))=T(z)\overline{T(w)}
\theta(z,w).
\end{equation}

\begin{lemma}\label{newlemma}
In the above notation, let
\begin{equation}\label{const1}
\phi(\xi,\eta)=a(\xi)\overline{a(\eta)}\phi_0(\xi, \eta)
\end{equation}
be the principal factorization of $\phi$. Then for some $c\in
\Com{}$, $c\ne 0$:
\begin{equation}\label{thth2}
\begin{split}
\theta(z, w)&=\frac{1}{|c|^2}\Res_\xi(f_z(\xi),
\Res_{\bar\eta}(\overline{f_w(\eta)},\phi_0(\xi,\eta))),\\
T(z)&=c\Res_\xi(f_z(\xi),a(\xi))^n.
\end{split}
\end{equation}
In particular, $\Res_\xi(f_z(\xi),\Res_{\bar\eta}
(\overline{f_w(\eta)},\phi_0(\xi,\eta)))$ is primitive.
\end{lemma}

\begin{proof}
Substituting (\ref{const1}) into (\ref{thth1}) and applying the
multiplicativity of the polynomial resultant we find
\begin{equation}\label{spis}
\begin{split}
T(z)\overline{T(w)}\theta(z,w)&=\Res_\xi(f_z(\xi),a(\xi)^n\cdot\Res_{\bar\eta}(
\overline{f_w(\eta)},\overline{a(\eta)})\cdot
\Res_{\bar\eta}(\overline{f_w(\eta)},\phi_0(\xi,\eta)))\\
&=h(z)^n\overline{h(w})^n\Res_\xi(f_z(\xi),
\Res_{\bar\eta}(\overline{f_w(\eta}),\phi_0(\xi,\eta))).
\end{split}
\end{equation}
Here $h(z)$ stands for the resultant $\Res_\xi(f_z(\xi),a(\xi))$
and by virtue of (\ref{conjj}) we have
$\overline{h(w})=\Res_{\bar\eta}(\overline{f_w(\eta)},\overline{a(\eta)})$.

By our assumption $\theta(z,w)$ is primitive. Hence we find from
(\ref{spis}) that
\begin{equation}\label{TT}
T(z)=h(z)^n t(z)
\end{equation}
for some polynomial $t(z)$. Therefore (\ref{spis}) yields
\begin{equation}\label{BY}
\Res_\xi(f_z(\xi),\Res_{\bar\eta}(\overline{f_w(\eta}),\phi_0(\xi,\eta)))
=t(z) \overline{t(w)}\theta(z,w),
\end{equation}
and, because $\theta(z,w)$ is primitive, (\ref{BY}) provides (up to a constant factor)
the principal  factorization for
the left hand side.

We claim now that $t(\xi)$ is equal to a constant. Indeed, to reach
a contradiction we assume that $\deg t(z)\geq 1$ and consider an
arbitrary root $\alpha$ of the polynomial $t(z)$. By virtue of
(\ref{BY}),
$$
\Res_\xi(f_\alpha(\xi),\Res_{\bar\eta}(\overline{f_w(\eta}),\phi_0(\xi,
\eta))) =0 \quad  (w\in \mathbb{C}).
$$
This means that polynomials $f_\alpha(\xi)=A(\xi)-\alpha B(\xi)$ and
$\Res_{\bar\eta}(\overline{f_w(\eta}),\phi_0(\xi, \eta))$ have a
common root for any $w$. Since $f_\alpha(\xi)$ does not depend on
$w$, a standard continuity argument yields that the common root
can be taken independently on $w$. Denote it by $\xi_0$. It follows
then that
\begin{equation}\label{vanish}
\Res_{\bar\eta}(\overline{f_w(\eta}),\phi_0(\xi_0, \eta))=0 \quad
(w\in\Com{}).
\end{equation}

Since $\phi_0(\xi,\eta)$ is primitive, by Lemma~\ref{lem:primitive},
we have $\phi(\xi_0,\eta)\not\equiv 0$. Then by virtue of
(\ref{vanish}),  $\phi(\xi_0,\eta)$ and $\overline{f_w(\eta})$ as
polynomials in $\overline{\eta}$ have a common root, say
$\overline{\eta_0}$, which again can be chosen independently of $w$.
Then
$$
0= {f_w(\eta_0)}=A(\eta_0)-w B(\eta_0) \quad  (w\in\Com{}).
$$
Hence $A(\eta_0)=B(\eta_0)=0$ which contradicts the assumption that
$A$ and $B$ are relatively prime. This  contradiction proves that
$t(z)$ is constant. Applying this to (\ref{TT}) we arrive at the
required formulas in (\ref{thth2}) and the lemma is proved.
\end{proof}

\begin{corollary}
Let $f_z(\zeta)=A(\zeta)-zB(\zeta)$ with $A$ and $B$ to be
relatively prime polynomials, $\deg A>\deg B$. Let $\phi(\xi,\eta)$
be a Hermitian polynomial such that
\begin{equation}\label{thth}
\Res_\xi(f_z(\xi),\Res_{\bar\eta}(\overline{f_w(\eta)},\phi(\xi,\eta)))
=T(z)\overline{T(w)}
\end{equation}
for some polynomial $T(z)$. Then $\phi(\xi,\eta)$ is separable,
i.e., there is a polynomial $a(\xi)$ such that
\begin{equation}\label{const}
\phi(\xi,\eta)=a(\xi)\overline{a(\eta)}.
\end{equation}
\end{corollary}

\begin{proof}
It suffices to show that the function $\phi_0(z,w)$ in
(\ref{const1}) is equal to a constant. By the first identity in
(\ref{thth2}) we have
$\Res_\xi(f_z(\xi),\Res_{\bar\eta}(\overline{f_w(\eta)},\phi_0(\xi,\eta)))\equiv
|c|^2 $ for some complex number $c\ne 0$. By the product formula
(\ref{res1}) this resultant, as a polynomial in $z$, has degree
$p\deg A$, where $p$ is the degree of
$\Res_{\bar\eta}(\overline{f_w(\eta)},\phi_0(\xi,\eta))$ as a
polynomial in $\xi$. Hence
$\deg_\xi\Res_{\bar\eta}(\overline{f_w(\eta)},\phi_0(\xi,\eta))=0$.
Since $\deg_{\bar \eta}\overline{f_w(\eta)}=\deg A\ne 0$, the same
argument shows that $\deg_\xi \phi_0(\xi,\eta)=0$.  But
$\phi_0(\xi,\eta)$ is Hermitian, hence it is a constant.
\end{proof}


\section{Proof of Theorem~\ref{th:mm}}
\label{sec:main}

We argue by contradiction and assume that, for some rational
function $f(\zeta)$ of degree $n=\deg f> p$, there is a domain
$\Omega$ such that $f$ is $p$-valent and proper in $\Omega$, and in
addition that the exponential transform of $\Omega$ is rational for $z$, $w$
large. Then, by
virtue of (\ref{www}),
\begin{equation}\label{tttt}
\Res^*_\xi (f(\xi)-z,\Res^*_{\bar
\eta}(\overline{f(\eta)}-\bar w,E_{\Omega}(\xi,\eta)))
=E_{\mathbb{D}}(z,w)^p=\left(\frac{z\bar w-1}{z\bar w}\right)^p,
\end{equation}
Since $E_{\Omega}(\xi,\eta)$ is rational we can write it as a fraction
$\frac{\phi(\xi,\eta)}{\psi(\xi,\eta)}$, where $\phi(\xi,\eta)$ and $\psi(\xi,\eta)$
are polynomials. By (\ref{Cauchy0}) we have  $E_{\Omega}(\xi,\infty)$=1 and, thus,
$\mathrm{ord}_{\eta=\infty} E_{\Omega}(\xi,\eta)=0$ for any
$\xi\in \mathbb{C}\setminus \overline{\Omega}$ (here $E_{\Omega}(\xi,\infty)$
is regarded as a rational function of $\eta$).  Hence we infer from (\ref{def1}) that
\begin{equation}\label{hh1}
h(\xi,w):=\Res^*_{\bar\eta}(\overline{f(\eta)}-\bar w,E_{\Omega}(\xi,\eta))
=\frac{\Res_{\bar\eta}(\overline{f(\eta)}-\bar w,\phi(\xi,\eta))}
{\Res_{\bar\eta}(\overline{f(\eta)}-\bar w,\psi(\xi,\eta))}.
\end{equation}

It easily follows from the Poisson product formula (\ref{res1})
that $h(\xi,w)$ is a rational function in  $\xi$ and $\bar w$. By
Lemma~\ref{lem:main part}, $E_\Omega$ is regular in the unbounded
component of $\Com{}\setminus \overline{\Omega}$ in the sense of
Definition~\ref{defreg}, hence
$$
\deg_\xi \phi(\xi,\eta)=\deg_\xi \psi(\xi,\eta)=:d.
$$

On the other hand, since $\deg_{\bar\eta} (\overline{f(\eta)}-\bar
w)=n$ independently of $w$ (recall that $\deg A>\deg B$), the
degrees of the numerator and the denominator in the right hand side
of (\ref{hh1}), as polynomials in $\xi$, are equal to $nd$. In
particular, $\ord_{\xi=\infty} h(\xi,w)=0$ and a not difficult
analysis of the leading coefficients of $\xi$ in the numerator and
denominator of $h(\xi,w)$ together with (\ref{appp}) shows that
$h(\infty,w)=1$ (alternatively, one can notice that the meromorphic
resultant
$\Res^*_{\bar\eta}(\overline{f(\eta)}-\bar
w,E_{\Omega}(\xi,\eta))$ in
the definition of $h$ is obviously a continuous function of $\xi\in
\mathbb{C}\setminus \overline{\Omega}$ and use that
$E_{\Omega}(\infty,\eta)=1$).

Summarizing these facts, we write the meromorphic resultant in (\ref{tttt})
by virtue of (\ref{def1}) and (\ref{res1}) in terms of polynomial resultants as
\begin{equation}\label{varr}
\begin{split}
\frac{\Res_\xi(f(\xi)-z,\Res_{\bar\eta}(\overline{f(\eta)}-\bar
w,\phi(\xi,\eta)))}
{\Res_\xi(f(\xi)-z,\Res_{\bar\eta}(\overline{f(\eta)}-\bar
w,\psi(\xi,\eta)))}
=z^{-p}\bar{w}^{-p}(z\bar w-1)^p.
\end{split}
\end{equation}
In the right hand side of (\ref{varr}) there is only one factor
which contains merely the variable $z$, namely $z^{-p}$. Now we look
for all factors of the left hand side of (\ref{varr}) which are
univariate polynomials in $z$. To this end, we pass to  the principal
factorizations
$$
\phi(\xi,\eta)=a(\xi)\overline{a(\eta})\phi_0(\xi,\eta), \qquad
\psi(\xi,\eta)=b(\xi)\overline{b(\eta})\psi_0(\xi,\eta),
$$
hence by multiplicativity of the resultant,
$$
\Res_{\bar\eta}(\overline{f(\eta)}-\bar w,\phi(\xi,\eta))=
a(\xi)^n\cdot \Res_{\bar\eta}(\overline{f(\eta)}-\bar w,\overline{a(\eta}))\cdot
\Res_{\bar\eta}(\overline{f(\eta)}-\bar w,\phi_0(\xi,\eta)),
$$
and the resultant in the numerator in (\ref{varr}) is found to be
the following product:
\begin{equation}\label{exactt}
\Res^n_\xi(f(\xi)-z,a(\xi))\cdot
\Res^n_{\bar\eta}(\overline{f(\eta)}-\bar w,\overline{a(\eta}))\cdot
\Res_\xi(f(\xi)-z,\Res_{\bar\eta}(\overline{f(\eta)}-\bar w,\phi_0(\xi,\eta)))
\end{equation}
The second factor in (\ref{exactt}) does not contain $z$ at all, and
the third factor is primitive by Lemma~\ref{newlemma}, hence it has no factors which depend on a single variable. It follows that the only factor
in (\ref{exactt}) which is a univariate polynomial in $z$ is
$\Res^n_\xi(f_z(\xi),a(\xi))$.

Repeating the same argument with the denominator in (\ref{varr}) and collecting all
factors which contain $z$ only, we arrive  at
\begin{equation}\label{varr1}
\begin{split}
\left(\frac{\Res_\xi(f_z(\xi),a(\xi))}
{\Res_\xi(f_z(\xi), b(\xi))}\right)^n=Cz^{-p}
\end{split}
\end{equation}
for some constant $C$. But the latter yields immediately that $n$
divides $p$, which contradicts our assumption $p<n.$ The theorem
follows.


\section{Appendix: the exponential transform of Bernoulli's lemniscate}
\label{sec:app}

Finally we treat the most classical lemniscate domain (or rather
open set), namely the set bounded by the lemniscate of Bernoulli
$$
\Omega=\{z\in \mathbb{C}:\;|z^2-1|<1\}.
$$
Obviously, the odd harmonic moments of $\Omega$ are zero and a
straightforward calculation for the even moments yields
$$
M_{2k}(\Omega)=\frac{2^{2k+1}(k!)^2}{\pi (2k+1)!}.
$$
Hence we obtain, for the corresponding Cauchy transform,
$$
C_{\Omega}(z)=\sum_{m\geq 0}\frac{M_{m}(\Omega)}{z^{m+1}}=
\frac{1}{\pi}\sum_{m\geq 0}\frac{(k!)^2}{ (2k+1)!}\left(\frac{2}{z}\right)^{2k+1}
=\frac{2\arcsin \frac{1}{z}}{\pi\sqrt{1-\frac{1}{z^2}}},
$$
which shows that $C_{\Omega}(z)$, and therefore also
$E_{\Omega}(z,w)$, is transcendental.

We find below a closed formula for the exponential transform of
$\Omega$. For any $p,q\geq 0$ with $p+q$ even, the $(p,q)$:th
harmonic moment is found by integration in polar coordinates:
\begin{equation*}
\begin{split}
M_{p,q}(\Omega)&=\frac{1}{\pi}\int_{\Omega}z^p\overline{z}^q \; dxdy =
\frac{2}{\pi}\int_{\Omega_+}z^p\overline{z}^q \; dxdy \\
&=\frac{2}{\pi}\int_{-\frac{\pi}{4}}^{\frac{\pi}{4}}e^{\I
(p-q)\theta}d\theta \int_{0}^{\sqrt{2\cos 2\theta}}
\rho^{p+q+1}d\rho\\
&=\frac{2^{\frac{p+q}{2}+2}}{\pi(p+q+2)}\int_{-\frac{\pi}{4}}^{\frac{\pi}{4}}
(\cos 2\theta)^{\frac{p+q}{2}+1}e^{\I (p-q)\theta}d\theta \\
&=\frac{2^{\frac{p+q}{2}+2}}{\pi(p+q+2)}\int_{0}^{\frac{\pi}{2}}
(\cos t)^{\frac{p+q}{2}+1}e^{\I \frac{p-q}{2}t}\,dt \\
&=\frac{2^{\frac{p+q}{2}+2}}{\pi(p+q+2)}
\int_{0}^{\frac{\pi}{2}}(\cos t)^{\frac{p+q}{2}+1}\cos ( \frac{p-q}{2}t)\,dt \\
\end{split}
\end{equation*}
where $\Omega_+=\Omega\cap \{z: \re z>0\}$ is the right petal of $\Omega$.
Expressing the last integral in terms of the Gamma function we obtain
\begin{equation}\label{we1}
\begin{split}
M_{p,q}(\Omega)
&=\frac{1}{2}\cdot\frac{\Gamma(\frac{p+q}{2}+1)}{\Gamma(\frac{p+1}{2}+1)
\Gamma(\frac{q+1}{2}+1)}.\\
\end{split}
\end{equation}

Let $p$ be an odd number, $p=2k+1$, $k\geq 0$. Then by the evenness of $p+q$, $q$
is odd too and we write $q=2m+1$. Hence
\begin{equation*}\label{we2}
\begin{split}
M_{2k+1,2m+1}(\Omega)
&=\frac{\Gamma(k+m+2)}{2\Gamma(k+2)\Gamma(m+2)}
=\frac{1}{2(k+m+2)}\binom{k+m+2}{k+1},\\
\end{split}
\end{equation*}
and we obtain for a partial sum
\begin{equation*}\label{we3}
\begin{split}
\sum_{k+m=n}\frac{M_{2k+1,2m+1}(\Omega)}{z^{2k+2}\bar{w}^{2m+2}}
&=\frac{1}{2(n+2)}\sum_{k=0}^n\binom{n+2}{k+1}(z^{-2})^{k+1}(\bar{w}^{-2})^{n+1-k}\\
&=\frac{1}{2(n+2)}\biggl(\biggl(\frac{1}{z^2}
+\frac{1}{\bar{w}^2}\biggr)^{n+2}-\frac{1}{z^{2(n+2)}}-\frac{1}{\bar{w}^{2(n+2)}}\biggr).
\end{split}
\end{equation*}
Therefore
\begin{equation*}\label{we4}
\begin{split}
S_{\mathrm{odd}}&\equiv \sum_{k,m\geq 0}
\frac{M_{2k+1,2m+1}(\Omega)}{z^{2k+2}\bar w^{2m+2}}
=\frac{1}{2}\sum_{n=0}^\infty \left(\frac{(z^{-2}
+\bar{w}^{-2})^{n+2}}{(n+2)}-\frac{z^{-2n-4}}{n+2}-\frac{\bar{w}^{-2n-4}}{n+2}\right)\\
&=\frac{1}{2}[\ln(1-z^{-2})+\ln(1-\bar{w}^{-2})-\ln(1-z^{-2}-\bar{w}^{-2})]\\
&=-\frac{1}{2}\ln\biggl(1-\frac{1}{(z^2-1)(\bar{w}^2-1)}\biggr),
\end{split}
\end{equation*}
and it follows from (\ref{infact}) that
\begin{equation}\label{foll}
E_\Omega(z,w)=\sqrt{1-\frac{1}{(z^2-1)(\bar{w}^2-1)}}\cdot \exp(-S_{\mathrm{even}}),
\end{equation}
where 
$$
S_{\mathrm{even}}\equiv \sum_{k,m\geq 0}
\frac{M_{2k,2m}(\Omega)}{z^{2k+1}\bar w^{2m+1}}.
$$

In order to find the even partial sum,  we find from (\ref{we1})
$$
M_{2k,2m}(\Omega)
=\frac{\Gamma(k+m+1)}{2\Gamma(k+\frac{3}{2})\Gamma(m+\frac{3}{2})}=
\frac{2}{\pi}\cdot \frac{(1)_{k+m}}{(\frac{3}{2})_{k}(\frac{3}{2})_m}
$$
where $(a)_{x}=\frac{\Gamma(a+x)}{\Gamma(a)}$ denotes the Pochhammer symbol. Thus
\begin{equation}\label{Seven}
\begin{split}
S_{\mathrm{even}}&=\frac{2}{\pi} \sum_{k,m\geq 0}
\frac{(1)_{k+m}}{(\frac{3}{2})_{k}(\frac{3}{2})_m}z^{-2k-1}\bar w^{-2m-1}\\
&=
\frac{2}{\pi z\bar w}\cdot F_2(1;1,1;\frac{3}{2},\frac{3}{2};z^{-2},\bar w^{-2}),
\end{split}
\end{equation}
where
$$
F_2(a;b,b';c,c';x,y)= \sum_{k,m= 0}^\infty
\frac{(a)_{k+m}(b)_k(b')_m}{(c)_k(c')_m}\,\frac{x^k}{k!}\,\frac{y^m}{m!}
$$
is the so-called Appell's function of the second kind
\cite[p.14]{Appell} (see also \cite[p.~53]{StrivaManocha}). It is
well-known \cite[p. 214, Eq. (8.2.3)]{Slater} that $F_2$, like the
hypergeometric function of Gauss, admits an integral
representations:
\begin{equation}\label{Slat}
F_2(a;b,b';c,c';x,y)= C
\int_{0}^1\int_0^1 \frac{(1-u)^{c-b-1}(1-v)^{c'-b'-1}}{u^{1-b}v^{1-b'}(1-xu-vy)^{a}}
\,du\, dv,
\end{equation}
where $C=\frac{\Gamma(c)\Gamma(c')}{\Gamma(b)\Gamma(b')\Gamma(c-b)\Gamma(c'-b')}$.
We find for our choice of  parameters
\begin{equation*}\label{inte}
F_2(1;1,1;\frac{3}{2},\frac{3}{2};z^{-2},\bar w^{-2})=\frac{1}{4}
\int_{0}^1\int_0^1 \frac{du\, dv}{\sqrt{1-u}\sqrt{1-v}(1-xz^{-2}-v\bar w^{-2})},
\end{equation*}
After an initial change of variables $u=1-\xi^2$ and $v=1-\eta^2$ we
find, after several additional changes of variables
\begin{equation*}\label{inte1}
\begin{split}
F_2(1;1,1;\frac{3}{2},\frac{3}{2};z^{-2},\bar w^{-2})&=
\int_{0}^1\int_0^1 \frac{d\xi\, d\eta}{(1-z^{-2}-\bar w^{-2})
+(\xi^2z^{-2}+\eta^2\bar w^{-2})}\\
&=\frac{1}{1-z^{-2}-\bar w^{-2}}\int_{0}^1\int_0^1
\frac{d\xi\, d\eta}{1+\xi^2s^{2}+\eta^2t^{2}}\\
&=\frac{1}{st(1-z^{-2}-\bar w^{-2})}\int_{0}^s\int_0^t
\frac{d\xi\, d\eta}{1+\xi^2+\eta^2}\\
&=z\bar w\int_{0}^s\int_0^t \frac{d\xi\, d\eta}{1+\xi^2+\eta^2},\\
\end{split}
\end{equation*}
where $s={z^{-1}}(1-z^{-2}-\bar w^{-2})^{-\frac{1}{2}}$ and
$t={\bar w^{-1}}(1-z^{-2}-\bar w^{-2})^{-\frac{1}{2}}$. By virtue (\ref{Seven})
this implies
$$
S_{\mathrm{even}}=\frac{2}{\pi}\int_{0}^s\int_0^t \frac{d\xi\, d\eta}{1+\xi^2+\eta^2}.
$$

\begin{remark}
Interesting to note that the right hand side in (\ref{Seven}) is the
well-known Hubbell Rectangular Source Integral, and it expresses the
response of an omni-directional radiation detector situated at
height $h=1$ directly over a corner of a plane isotropic rectangular
(plaque) source of length $1/z$, width $1/w$ and a constant uniform
strength \cite{Hub1}, \cite{Hub2} (see also \cite{Kal}).
\end{remark}

In general, for the rose-lemniscate $\Omega_n=\{z\in \mathbb{C}:|z^n-1|<1\}$,
a similar argument shows that
\begin{equation}\label{rose1}
M_{kn+\lambda,mn+\lambda}(\Omega_n)=\frac{1}{n}\;
\frac{\Gamma(k+m+\frac{2(1+\lambda)}{n})}{\Gamma(k+1+\frac{1+\lambda}{n})\,
\Gamma(m+1+\frac{1+\lambda}{n})},
\end{equation}
when $\lambda=0,1,\ldots, p-1$, and $M_{ij}(\Omega_n)=0$ for $i-j\not \equiv 0\mod n$.
After a series of simple transformations this yields
$$
\sum_{i,j}\frac{M_{ij}(\Omega_n)}{z^{i+1}\bar w ^{j+1}}
=\frac{1}{n}\,\sum_{\lambda=0}^{p-1}
\,(z\bar w)^{-1-\lambda} \,S_\lambda(z^{-n},w^{-n}),
$$
where
$$
S_\lambda(x,y)=\frac{\Gamma(\frac{2(1+\lambda)}{n})}{\Gamma(\frac{1+n+\lambda}{n})^2}\,
F_2\biggl(\frac{2\lambda+2}{n};1,1;\, \frac{\lambda+n+1}{n},
\frac{\lambda+n+1}{n};\,x,y\biggr)
$$
is an Appell function of the second kind. Applying first a fractional linear transformation formula (8.3.10) in \cite[p.~219]{Slater}
\begin{equation}\label{Slat2}
F_2(a;b,b';\,c,c';\, x,y)=\frac{F(a;c-b,c'-b';\,c,c';\,
\frac{x}{x+y-1},\frac{y}{x+y-1})}{(1-x-y)^{a}}
\end{equation}
and then (\ref{Slat}), we get
\begin{equation}
\begin{split}
S_\lambda(x,y)&=\frac{\Gamma(2\lambda_n)}{\Gamma(1+\lambda_n+1)^2}\;
\frac{F_2(2\lambda_n;\lambda_n,\lambda_n;\, \lambda_n+1,\lambda_n+1;
\,\frac{x}{x+y-1},\frac{y}{x+y-1})}{(1-x-y)^{2\lambda_n}}\\
&=\frac{\Gamma(2\lambda_n)}{\Gamma(\lambda_n)^2}
\int_{0}^1\int_0^1 \frac{(1-x'-y')^{2\lambda_n}\,du\, dv,}
{u^{1-\lambda_n}v^{1-\lambda_n}(1-x'u-vy')^{2\lambda_n}}\\
&=\frac{\Gamma(2\lambda_n)}{\Gamma(1+\lambda_n)^2}
\int_{0}^1\int_0^1 \frac{(1-x'-y')^{2\lambda_n}\,d\xi\, d\eta,}
{(1-x'\xi^{1/\lambda_n}-y'\eta^{1/\lambda_n})^{2\lambda_n}}\\
\end{split}
\end{equation}
where $\lambda_n=\frac{1+\lambda}{n}$, $x'=\frac{x}{x+y-1}$ and $y'=\frac{y}{x+y-1}$.




\end{document}